\documentclass[11pt,reqno]{amsart}

\usepackage{amsmath}
\usepackage{amssymb}
\usepackage{amsbsy}
\usepackage{amsthm}
\usepackage{amsfonts}
\usepackage{enumerate}

\usepackage[bitstream-charter]{mathdesign}
\usepackage[T1]{fontenc}

\usepackage{esint}

\newtheorem{thm}{Theorem}
\newtheorem{lem}[thm]{Lemma}

\newtheorem{defn}{Definition}
\newtheorem{rem}[thm]{Remark}
\newtheorem{nota}{Notation}

\usepackage{color}

\newcommand{\N}{\mathbb{N}}
\newcommand{\R}{\mathbb{R}}

\newcommand{\bR}{\mathbb{R}}

\newcommand{\cA}{\mathcal{A}}
\newcommand{\cC}{\mathcal{C}}

\newcommand{\cE}{\mathcal{E}}
\newcommand{\cF}{\mathcal{F}}
\newcommand{\cH}{\mathcal{H}}

\newcommand{\cM}{\mathcal{M}}
\newcommand{\cP}{\mathcal{P}}
\newcommand{\cS}{\mathcal{S}}

\newcommand{\norm}[1]{{\left\Vert#1\right\Vert}}
\newcommand{\set}[1]{{\left\lbrace #1 \right\rbrace}}

\DeclareMathOperator*{\divg}{div}
\DeclareMathOperator*{\BMO}{BMO}
\DeclareMathOperator*{\dist}{dist}
\DeclareMathOperator*{\ess}{ess}
\DeclareMathOperator*{\loc}{loc}

\newcommand{\wL}{L^{3,w}(\Omega)}
\newcommand{\intzi}{\int_{0}^{\infty}}
\newcommand{\snzi}{\sum_{n=0}^{\infty}}

\newcommand{\vep}{\varepsilon}

\newcommand{\m}{\boldsymbol m}

\begin{document}
\baselineskip=18pt

\title[Regularity of the Navier--Stokes equations]{Local kinetic energy and singularities of the incompressible Navier--Stokes Equations}

\author{Hi Jun Choe \& Minsuk Yang}

\address{H. J. Choe: Department of Mathematics, 
Yonsei University,
Yonseiro 50, Seodaemungu Seoul, Korea}
\email{choe@yonsei.ac.kr}

\address{M. Yang: Korea Institute for Advanced Study \\
Hoegiro 85, Dongdaemungu Seoul, Korea}
\email{yangm@kias.re.kr}

\begin{abstract}
We study the partial regularity problem of the incompressible Navier--Stokes equations. 
In this paper, we show that a reverse H\"older inequality of velocity gradient with increasing support holds under the condition that a scaled functional corresponding the local kinetic energy is uniformly bounded.
As an application, we give a new bound for the Hausdorff dimension and the Minkowski dimension of singular set when weak solutions $v$ belong to $L^\infty(0,T;L^{3,w}(\mathbb{R}^3))$ where $L^{3,w}(\mathbb{R}^3)$ denotes the standard weak Lebesgue space.
\noindent{\bf keywords:} 
Navier--Stokes equations,
Partial regularity,
Minkowski dimension,
Singular set
\end{abstract}

\maketitle

\section{Introduction}
\label{S1}

In this paper we study the singular points of suitable weak solutions 
to the three dimensional incompressible Navier--Stokes equations.
Although general boundary and geometric conditions are important, 
we consider merely an initial value problem of the Navier-Stokes equations in the whole space $\Omega_T=\bR^3 \times (0,T)$:
\begin{equation}
\label{E11}
\begin{split}
(\partial_t - \nu\Delta) v + \divg (v\otimes v) +\nabla p &= 0 \\
\divg v &= 0
\end{split}
\end{equation}
with the initial data $v(x,0)= v_0(x)$, 
where $v$ and $v_0$ are three dimensional solenoidal vector fields 
and the pressure $p$ is a scalar field. 
In this paper, we let the viscosity $\nu=1$ since it is not important in our regularity analysis. 
We denote by $z=(x,t)$ space-time points, space balls by $B(x,r)=\set{y\in\bR^3 : |y-x| <r}$, and parabolic cylinders by
\[Q(z,r)=B(x,r)\times (t-r^2,t+r^2).\]
We always assume that parabolic cylinders are in the space-time domain and suppress reference points $z$ in various expressions when it can be understood obviously in the context.
The precise definitions will be given in the next section.

It is not known whether the solution stays regular for all time although the global existence of weak solutions was proved by Leray \cite{L} long ago.
To study the regularity problem Scheffer \cite{Sch} introduced the concept of suitable weak solutions and proved a partial regularity result. 
From Scheffer's result, one can conclude that the Minkowski dimension of the singular set is not greater than 2.
Caffarelli--Kohn--Nirenberg \cite{CKN} proved that the Hausdorff dimension of the singular set is not greater than 1 using a regularity criterion based on a scaled invariant functional corresponding to the velocity gradient.
Lin \cite{MR1488514} gave a simplified short proof by a blowup argument.
Ladyzhenskaya and Seregin \cite{MR1738171} gave a clear presentation of the H\"older regularity. 
Choe and Lewis \cite{MR1780481} studied the singular set by using a generalized Hausdorff measure.
Gustafson, Kang and Tsai \cite{MR2308753} unified several known criteria.
One of the most important conditions to guarantee the regularity of weak solutions is the so-called Ladyzhenskaya--Prodi--Serrin \cite{MR0236541,MR0126088,MR0150444} condition, that is,
\[u \in L^l(0,T;L^s(\R^3))\]
for some $s$ and $l$ satisfying $\frac{3}{s} + \frac{2}{l} = 1$ and $3 < s \le \infty$.
Escauriaza--Sergin--\v{S}ver\'{a}k \cite{MR2005639} resolved the regularity problem for the marginal case $v \in L^\infty(0,T;L^3(\R^3))$.
There are many variations and generalizations of the LPS condition including the Lorentz spaces.
But, most of them dealt with the case that the space integrability exponent is greater than 3.
The marginal case $v \in L^\infty(0,T;\wL)$ of the Lorentz spaces was studied by Kozono \cite{Ko} and Kim--Kozono \cite{KK}, where $\wL$ denotes the standard weak Lebesgue space.
They obtained some regularity results when the norm $\norm{v}_{L^\infty(0,T;\wL)}$ is sufficiently small.
Recently, there arise many interests on this much weaker case $v \in L^\infty(0,T;\wL)$ related to essential singularities without any smallness condition.

In studying the regularity problem of the Navier--Stokes equations, many papers dealt with the quantities in terms of gradients of solution s like $\int_Q |\nabla v|^2$ or $\int_Q |\omega|^2$ where $\omega$ denote the vorticity $\omega=\nabla\times v$.
In this paper, we study the regularity problem based on  local kinetic energy $\int_{B(x,R)} |v(y,t)|^2 dy$.
Under the condition that a scaled functional corresponding the local kinetic energy is uniformly bounded, we show that a reverse H\"older inequality of $\nabla v$ with increasing support holds.
Here is our first theorem.

\begin{thm}
\label{T11}
Suppose there exists a number $M \geq 1$ such that for all $Q(z,R) \subset \Omega_T$
\begin{equation}
\label{E12}
\sup_{|s-t|<R^2} R^{-1} \int_{B(x,R)} |v(y,s)|^2 dy \leq M.
\end{equation}
Then there exist positive constants $C$ and $\delta_0 = \frac{2}{CM^{5/9}-1}$ such that for any $0 \le \delta < \delta_0$ 
\begin{equation}
\label{E13}
\norm{\nabla v}_{L^{2+\delta}(Q(z,1))} \leq \overline{C} (\norm{\nabla v}_{L^{2}(Q(z,2))} +1)
\end{equation}
where $\overline{C} = (1+\frac{\delta}{2}-\frac{\delta}{2}CM^{5/9})^{-1}$.
\end{thm}

We say that a space-time point $z=(x,t)$ is singular if $v$ is locally unbounded at $z$.
We denote by $\cS$ the set of all singular points of $v$.
It is well-known that $\cS$ is a compact set.
Using this theorem we can gain better information about the distribution of the singular set $\cS$ if a weak solution $v$ of the Navier--Stokes equations belongs to $L^\infty(0,T;\wL)$.
In a recent paper \cite{CWY}, it is studied that if a weak solution to the Navier--Stokes equations belongs to $L^\infty(0,T; \wL)$, then the number of possible singular points at a singular time is finite.
In this paper, we focus on the fractal dimension of the singular set.
The following theorem is actually an immediate corollary.

\begin{thm}
\label{T12}
If $v \in L^\infty(0,T; \wL)$ is a weak solution to the Navier--Stokes equations, then the Hausdorff dimension of $\cS$ is also at most $1-2\delta_0$ 
\begin{equation}
\label{E14}
\dim_{\cH}(\cS) \le 1-2\delta_0
\end{equation}
\end{thm}

There are several different ways measuring lower dimensional sets like $\cS$. 
The concept of the Hausdorff dimension is widely used because the Hausdorff measure is
a natural generalization of the Euclidean Lebesgue measures.
Another important one is the Minkowski (box-counting) dimension. 
In general, to improve the Minkowski dimension is harder than the Hausdorff dimension. 
For the usual suitable weak solutions, the bound on the Hausdorff dimension is 1.
But, the bound on the Minkowski dimension remains significantly greater than 1.
There are several attempts to lower the Minkowski dimension of $\cS$, for example, \cite{MR2967124,Ku,MR2864801,KY}.
One of the main result of this paper is to show that if a weak solution $v$ belongs to $L^\infty(0,T; \wL)$, then we have a better bound for the Minkowski dimension of singular set of $v$.
Here is our last theorem.

\begin{thm}
\label{T13}
If $v \in L^\infty(0,T; \wL)$ is a weak solution to the Navier--Stokes equations, then the parabolic upper Minkowski dimension 
\begin{equation}
\label{E15}
\overline{\dim}_{\cM}(\cS) \le \frac{9-18\delta_0}{7-6\delta_0}
\end{equation}
where $\delta_0$ is the same constant in Theorem \ref{T11}.
\end{thm}

We remark that it is very hard to find the optimal constant $\delta_0$ in Theorem \ref{T11}.
If the number $\delta_0$ is very small, then the bound in Theorem \eqref{T13} is close to $9/7$.

We briefly explain the idea of proofs.
If a weak solution $v$ is in $L^\infty(0,T; \wL)$, then the local kinetic energy of $v$ is uniformly bounded.
Thus, we apply Theorem \ref{T11} to obtain the reverse H\"older inequality of $\nabla v$, which immediately implies the improvement of the Hausdorff dimension of the singular set.
To obtain a bound for the Minkowski dimension of the singular set, we 
need an interpolation argument to get the local higher integrability of $v$.
Combining the pressure decomposition and a covering argument, we finally obtain a bound for the Minkowski dimension of $\cS$.

\section{Preliminaries}
\label{S2}

In this section we recall basic definitions and set up our notations.
We denote by $\N$ the set of natural numbers and by $\R$ the set of real numbers.

\begin{defn}[Function spaces]
For $0<p\le\infty$ we denote by $L^p(\Omega)$ the standard Lebesgue spaces and by $W^{1,p}(\Omega)$ the standard Sobolev space with the norm 
\[\norm{f}_{W^{1,p}(\Omega)} = \left(\int_{\Omega} |f|^p +|\nabla f|^p dx\right)^{1/p}.\]
We denote $H^1(\Omega) = W^{1,2}(\Omega)$ and define $H_0^1(\Omega)$ to be the $H^1$ closure of $C^\infty_0(\Omega)$.
We denote the parabolic space  
\begin{equation}
V(\Omega_T)= L^{\infty}(0,T:L^2(\bR^3)) \cap L^2(0,T:H^1_0(\bR^3)).
\end{equation}
where 
\begin{align*}
L^{\infty}(0,T: L^2(\Omega)) &= \set{f : \ess \sup_{0<t<T} \norm{f}_{L^2(\Omega)}(t) < \infty} \\
L^q(0,T:W^{1,p}(\Omega)) &= \set{f : \int_0^T \norm{f}^q_{W^{1,p}(\Omega)}(t)dt <\infty}.
\end{align*}
\end{defn}

We now recall the definition of the weak Lebesgue spaces.
For a measurable function $f$ on $\R^3$, its level set with the height $h$ is denoted by
\begin{equation}
\label{E21}
E(h) = \set{x \in \R^3:|f(x)|>h}.
\end{equation}
The Lebesgue integral can be expressed by the Riemann integral of such level sets.
In particular, for $0<q<\infty$
\begin{equation}
\label{E22}
\int |f(x)|^q dx = \intzi qh^{q-1} \m(E(h)) dh.
\end{equation}

\begin{defn}[Weak Lebesgue space]
The weak Lebesgue space $L^{q,w}(\R^3)$ is the set of all measurable function such that the quantity
\begin{equation}
\label{E23}
\norm{f}_{q,w} := \sup_{h>0} \big[h \m(E(h))^{1/q}\big]
\end{equation}
is finite.
\end{defn}

We recall the definition of suitable weak solutions (see also \cite{CKN} and \cite{MR1488514}).

\begin{defn}[Suitable weak solutions]
Let $v_0 \in L^2(\bR^3)$ denote a given initial data, which is weakly divergence free vector field.
We say that $(v,p)\in V(\Omega_T) \times L^{3/2}(\Omega_T)$ is a suitable weak solution to the initial value problem if for all $\phi\in C^\infty_0(\Omega_T)$
\begin{equation}
\int \left(v\cdot \partial_t\phi - \nabla v:\nabla\phi + v\otimes v : \nabla \phi + p \nabla\cdot\phi\right) dz = 0
\end{equation}
where $dz=dxdt$.
The vector field $v$ is weakly divergence free for almost all time and satisfies
the localized energy inequality 
\begin{equation}
\label{E24}
\begin{split}
&\int |v(x,t)|^2 \phi dx + 2\int_0^t \int |\nabla v|^2\phi dz \\
&\le \int_0^t \int |v|^2 (\partial_t\phi+\Delta\phi) dz
+ \int_0^t \int (|v|^2+2p) v \cdot \nabla \phi dz
\end{split}
\end{equation}
for almost all $0 < t \le T$ and for all non-negative test functions $\phi \in C_0^{\infty}(\Omega_T)$ and
\[\int |v(x,t)-v_0(x)|^2 dx \to 0 \quad \text{as} \quad t \to 0.\]
\end{defn}

\begin{defn}[The parabolic Hausdorff dimension]
For fixed $\rho>0$ and $S\subset\bR^3\times\bR$, let $\cC(S,\rho)$ be the family of all coverings of parabolic cylinders $\set{Q(z_j,r_j)}$ that covers $S$ with $0<r_j \le \rho$.
The $\alpha$ dimensional parabolic Hausdorff measure is defined as
\[\cH^\alpha(S) = \lim_{\rho\to0} \inf_{\cC(E,\rho)} \sum_j r_j^\alpha.\]
The parabolic Hausdorff dimension of the set $S$ is defined as
\[\dim_{\cH}(S) = \inf\set{\alpha : \cH^\alpha(S)=0}.\]
\end{defn}

\begin{defn}[Singular points]
We call a space-time point $z=(x,t)$ is a singular point of a suitable weak solution $(v,p)$ if $v$ is not essentially bounded in any neighbourhood $Q(z,r)$.
We denote the set of all singular points by $\cS$.
\end{defn}

\begin{defn}[The parabolic upper Minkowski dimension]
Let $N(S;r)$ denote the minimum number of parabolic cylinders $Q(z,r)$ required to cover the set $S$.
Then the parabolic upper Minkowski dimension of the set $S$ is defined as
\begin{equation}
\label{E25}
\overline{\dim}_{\cM}(S) = \limsup_{r\to0} \frac{\log N(S;r)}{-\log r}.
\end{equation}
\end{defn}

In general, different fractal dimensions reflect the geometric structure of the set.
The upper Minkowski dimension is strongly control the Hausdorff dimension.
Indeed, from the definition it is easy to see that 
\begin{equation}
\dim_\cH(S) \le \overline{\dim}_\cM(S).
\end{equation}
We refer the reader to Falconer's monograph \cite{MR3236784} for the comprehensive introduction of the fractal geometry.

\begin{nota}
\begin{itemize}
\item
We denote $\R_+=\set{x\in\R:0<x<\infty}$.
\item
We use the following shorthand notation for balls and cylinders
\[RB = B(0,R), \quad RQ = Q(0,R).\]
\item
The average value of $f$ over $X$ is denoted by $(f)_E = \fint_E f d\m = \m(E)^{-1} \int_X f d\m$ where $\m(E)$ denotes the Lebesgue measure of the set $E$.
When the center of ball is obvious in the context, we write 
\[(f)_R = (f)_{B(x,R)}.\]
\item
We write $X \lesssim Y$ if there is a generic positive constant $C$ such that $|X| \le C|Y|$. 
\end{itemize}
\end{nota}

\section{Auxiliary lemmas}
\label{S3}

It is easy to see that the pressure satisfies
\[-\Delta p = \partial x_i\partial x_j (v_i v_j)\]
in the sense of distribution.
From this pressure equation, we have the following explicit decomposition of the localized pressure, which was presented in \cite{CKN}.

We choose $\psi$ satisfying $\psi(y) = 1$ for $y \in B(x,(\rho+r)/2)$ and $\psi(y) = 0$ for $y \notin B(x,r)$ so that  $p_3$ is harmonic in $B(x,\rho)$. 
Let $\overline{v} = v - (v)_r$.
We have the decomposition
\begin{equation}
\label{E31}
p\psi = p_1 + p_2 + p_3
\end{equation}
where 
\begin{equation}
\label{E32}
\begin{split}
p_1(x,t)
&=|\overline{v}|^2\psi(x,t)+ \frac{3}{4\pi} \int \partial_{y_i} \partial_{y_j} \left(\frac{1}{|x-y|}\right)(\overline{v}_i \overline{v}_j\psi)(y,t) dy, \\
p_2(x,t) &= \frac{3}{2\pi} \int \frac{x_i-y_i}{|x-y|^3}
(\overline{v}_i\overline{v}_j\partial_j\psi)(y,t) dy
+ \frac{3}{4\pi} \int \frac{1}{|x-y|} (\overline{v}_i\overline{v}_j\partial_i\partial_j\psi)(y,t) dy, \\
p_3(x,t) &= \frac{3}{4\pi} \int \frac{1}{|x-y|} (p\Delta \psi)(y,t) dy 
+ \frac{3}{2\pi} \int \frac{x_i-y_i}{|x-y|^3} (p\partial_i\psi)(y,t) dy.
\end{split}
\end{equation}
We notice that $p_1 + p_2$ depends only on $v$ and that there is no improvement in integrability of $p_3$ with respect to time even if $v$ has higher integrability in time, as is observed in Serrin's example in \cite{SERR}.

\begin{lem}
\label{T31}
If $v \in L^{\infty}(t-r^2,t+r^2;L^{2q}(B(x,r)))$  for some $q \in
(1,\infty)$, then
\[p_1 + p_2 \in L^{\infty} (t-r^2,t+r^2;L^{q}(B(x,r)))\]
satisfies
\begin{equation}
\label{E33}
\sup_{|t-t_0|< r^2} \norm{p_1 + p_2}_{L^q(B_r)} \lesssim  \sup_{|t-t_0|<r^2} \norm{\overline{v}}^2_{L^{2q}(B_r)}
\end{equation}
and $p_3$ satisfies
\begin{equation}
\label{E34}
\begin{split}
&\norm{p_3}_{L^{3/2}(t_0-\rho^2,t_0+\rho^2; L^{\infty}(B_\rho))} + (r-\rho) \norm{\nabla p_3}_{L^{3/2}(t_0-\rho^2,t_0+\rho^2; L^{\infty}(B_\rho) )} \\
&\lesssim \frac{1}{|r-\rho|^3}
\norm{p}_{L^{3/2} (t_0-r^2,t_0+r^2; L^{1}(B_r\setminus B_\rho))}.
\end{split}
\end{equation}
\end{lem}

\begin{proof}
The estimate \eqref{E33} follows from the Calderon--Zygmund estimate for $p_1$ and potential estimates for $p_2$.
On the contrary, since $p_3$ is harmonic in $B(x,(\rho+r)/2)$, the mean value property gives 
\[\norm{p_3}_{L^{\infty}(B(x,\rho))} 
+ (r-\rho) \norm{\nabla p_3}_{L^{\infty}(B(x,\rho))}
\lesssim \frac{1}{|r-\rho|^3} \norm{p}_{L^{1}(B(x,r) \setminus B(x,\rho))}.\]
Integrating with respect to time yields the estimate \eqref{E34}.
\end{proof}

We modify localized energy inequality in terms of average free
velocity like Lemma 2.1 in Seregin\cite{SER}, which was also known
in parabolic equations (see also Choe \cite{C} and Giaquinta--Struwe \cite{GS}).

\begin{lem}
\label{T32}
Let $\phi \in C^\infty_{0}(B_r(x_0))$ and $\theta\in C^\infty_0(t_0-r^2,t_0+r^2)$ be non-negative cutoff functions with $\int\phi dx =1$.
If we denote the weighted space average $[v]_{r}(t) = \int  v(x,t)\phi(x)dx$  and $\bar{v} = v-[v]_r$, then 
\begin{equation}
\label{E35}
\begin{split}
&\sup_t \int |\bar{v}|^2 \phi\theta dx 
+ 2 \int |\nabla v|^2 \phi\theta dxdt \\
&\le \int |\bar{v}|^2(\Delta\phi\theta+\phi\theta_t) dxdt 
+ \int v\cdot\nabla\phi|\bar{v}|^2 \theta dxdt 
+ 2 \int p \bar{v}\cdot\nabla\phi\theta dxdt.
\end{split}
\end{equation}
\end{lem}

\begin{proof}
We may assume $x_0=0$.
Integrating in space we find that
\[\frac{d}{dt} [v]_r(t) =- \int \nabla v \cdot \nabla\phi dx +\int
v\cdot\nabla\phi \bar{v} dx +\int p \nabla\phi dx.\]
In the localized energy inequality \eqref{E24} we take a cutoff
function $\phi(x)\theta(t)$. 
Then 
\begin{align}
\label{E36}
\int |v|^2 \phi\theta(t_1) dx 
&= \int_{B_r} |\bar{v}|^2\theta dx + |[v]_r|^2\theta(t_1) \\
\label{E37}
\int_{t_0-r^2}^{t_1}\int |v|^2 \phi \theta_t dxdt
&= \int_{t_0-r^2}^{t_1}\int |\bar{v}|^2 \phi \theta_t dxdt
+ \int_{t_0-r^2}^{t_1} |[v]_r|^2  \theta_t dt.
\end{align}
Integrating by parts, we also have
\begin{equation}
\label{E38}
\begin{split}
\int_{t_0-r^2}^{t_1} |[v]_r|^2 \theta_t dt 
&= |[v]_r|^2 \theta(t_1) - 2 \int_{t_0-r^2}^{t_1} [v]_r\cdot\frac{d}{dt}[v]_r \theta dt \\
&= |[v]_r|^2 \theta(t_1) + 2 \int_{t_0-r^2}^{t_1} \int [v]_r \cdot\nabla {v}\cdot \nabla\phi \theta dxdt \\
&\quad - 2 \int_{t_0-r^2}^{t_1} \int v\cdot \nabla \phi \bar{v} \cdot  [v]_r \theta dx dt 
- 2 \int_{t_0-r^2}^{t_1} \int p [v]_r \cdot \nabla \phi \theta dx dt.\end{split}
\end{equation}
Since $t_1$ is arbitrary, combining \eqref{E24}, \eqref{E36}, \eqref{E37} and \eqref{E38}, we get an average free localized energy inequality.
 \end{proof}
\begin{rem} In the weighted average free localized inequality \eqref{E35}, we can replace
the weighted space average $[v]_r$ by the space average $(v)_r$.
\end{rem}

\section{Proof of Theorem \ref{T11}} 
\label{S4}

We begin by recalling the following iteration lemma, which can be found in M. Giaquinta \cite{G}.

\begin{lem}
\label{T41}
Let $f(r)$ be a non-negative bounded function on $[R_0, R_1] \subset \R_+$.
If there are non-negative constants $A, B, C$ and positive exponents $b<a$ and a parameter $\theta \in (0,1)$ such that for all $R_0 \le r < R \le R_1$
\[f(r) \le \theta f(R) + A(R-r)^{-a} + B(R-r)^{-b} + C,\]
then for all $R_0 \le r < R \le R_1$
\[f(r) \lesssim_{a,\theta} A(R-r)^{-a} + B(R-r)^{-b} + C.\]
\end{lem}

\begin{proof}
For reader's convenience, we give a sketch of the proof.
Fix $r$ and $R$.
Define $r_0=r$ and 
\[r_{n+1} = r_n + (R-r) d_n\]
where $d_n$ is a sequence of positive numbers with $\snzi d_n=1$ so that 
$\lim_{n\to\infty} r_n = R$.
From the condition 
\[\theta^n f(r_n) - \theta^{n+1} f(r_{n+1}) \le \set{A(R-r)^{-a} d_n^{-a} + B(R-r)^{-b} d_n^{-b} + C} \theta^n.\]
Since $d_n\le1$ and $b<a$, we have $1 \le d_n^{-b} \le d_n^{-a}$ and so
\[f(r) \le \set{A(R-r)^{-a} + B(R-r)^{-b} + C} \snzi \theta^n d_n^{-a}.\]
It suffices to find $d_n$ such that $\snzi d_n=1$ and 
\[\snzi \theta^n d_n^{-a} < \infty.\]
For example, we can choose $\tau$ satisfying $\theta < \tau^a < 1$ and set $d_n = (1-\tau) \tau^n$.
\end{proof}

In order to neatly describe an iteration scheme, we introduce the following functionals. 

\begin{defn} 
Let $z=(x,t)$ and $\hat{v} = v - (v)_{R}$ and define 
\begin{align*}
\cA(z,R) &= \sup_{|t|<R^2} \frac{1}{R^5} \int_{B(x,R)}|\hat{v}|^2 dx \\
\cE(z,R) &= \fint_{Q(z,R)} |\nabla v|^2 dz, \\
\cF(z,R) &=\left(\fint_{Q(z,R)}|\nabla v|^{9/5} dz\right)^{10/9} \\
\cP(z,R) &= \fint_{Q(z,R)} |\nabla p_3 \cdot v| dz.
\end{align*}
\end{defn}

Now, we are ready to prove Theorem \ref{T11}.
We divide its proof into severl steps.
\begin{enumerate}[\bf{Step} 1)]
\item
We may assume the reference point $z=0$ and suppress it for the notational convenience because one can easily see that all the estimates in the following proof does not depend on the reference point $z$.
Assuming \eqref{E12} for all $Q(z,r) \subset \Omega_T$ we shall prove first that $\nabla v$ satisfies a reverse H\"older inequality.
More precisely, under the assumption \eqref{E12} with $M\ge1$, there exists an absolute positive constant $C$ such that
\begin{equation}
\label{E41}
\fint_{Q/2}|\nabla v|^2 dz
\le C M^{5/9} \left(\fint_{Q} |\nabla v|^{9/5} dz\right)^{10/9}
+ C \fint_{Q} |\nabla p_3\cdot{v}| dz
\end{equation}
for every cube $Q \subset \Omega_T$.
Using the shorthand notation for functionals we rewrite it as 
\begin{equation}
\label{E42}
\cE(R/2) \le CM^{5/9} \cF(R) + C\cP(R).
\end{equation}
The constant $C$ does not depend on $M$, $R$, and the suppressed reference point $z$.
\item
Fix $0<r<R$.
Choose a smooth cutoff function $\phi$ satisfying $\phi(z)=1$ for $z \in rQ$, $\phi(z)=0$ for $z\notin RQ$, and for all $k \in \set{0} \cup \N$
\[|\nabla^k\phi| \lesssim (R-r)^{-k}, \quad
|\nabla^k \partial_t \phi| \lesssim (R-r)^{-k-2}.\]
From the localized energy inequality \eqref{E35} replacing the weighted space average $[v]_r$ to space average $(v)_r$, we obtain
that
\begin{equation}
\label{E43}
\begin{split}
\cA(r) + \cE(r)
&\lesssim (R-r)^{-2} \fint_{RQ} |\hat{v}|^2 dz
+ (R-r)^{-1} \fint_{RQ} |v| |\hat{v}|^2 dz \\
&\quad + (R-r)^{-1} \fint_{RQ} |p_1 + p_2| |\hat{v}| dz
+ \fint_{RQ} |\nabla p_3\cdot\hat{v}| dz \\
&=: I_1 + I_2 + I_3 + \cP(R).
\end{split}
\end{equation}
\item
We estimate the term $I_1$.
We use the Sobolev inequality to get 
\begin{align*}
I_1 
&= (R-r)^{-2} \fint_{RQ} |\hat{v}|^2 dz \\
&\lesssim (R-r)^{-2} R^{-1/2} R^{-9/2} \int_{R^2I} \left(\int_{RB} |\hat{v}|^2 dx\right)^{1/10} \left(\int_{RB} |\hat{v}|^2 dx\right)^{9/10} dt \\
&\lesssim (R-r)^{-2} \cA(R)^{1/10} R^{-1/2} \int_{R^2I} \left(\int_{RB} |\nabla v|^{6/5} dx\right)^{3/2} dt.
\end{align*}
By the Jensen inequality 
\[\int_{R^2I} \left(\int_{RB} |\nabla v|^{6/5} dx\right)^{3/2} dt 
\lesssim R^{3/2} \int_{RQ} |\nabla v|^{9/5} dx 
\lesssim R^{13/2} \cF(R)^{9/10}.\]
Hence, we get by the Young inequality 
\begin{equation}
\label{E44}
\begin{split}
I_1 
&\lesssim (R-r)^{-2} R^2 \cA(R)^{1/10} \cF(R)^{9/10} \\
&\lesssim \vep \cA(R) + \vep^{-1/9}  (R/(R-r))^{20/9} \cF(R).
\end{split}
\end{equation}
\item
We estimate the term $I_2$.
Using the assumption \eqref{E12} and the Sobolev inequality, we obtain
\begin{align*}
\int_{RB} |v| |\hat{v}|^2 dx
&\le \left(\int_{RB} |v|^2 dx\right)^{1/2} \left(\int_{RB} |\hat{v}|^2 dx\right)^{1/10}
\left(\int_{RB} |\hat{v}|^{9/2} dx\right)^{2/5} \\
&\lesssim (R M)^{1/2} \left(\int_{RB} |\hat{v}|^2
dx\right)^{1/10} \int_{RB} |\nabla v|^{9/5} dx.
\end{align*}
Integrating in time, we get
\begin{align*}
I_2 
&= (R-r)^{-1} \fint_{RQ} |v| |\hat{v}|^2 dz \\
&\lesssim (R-r)^{-1} R M^{1/2} \cA(R)^{1/10} \fint_{RQ} |\nabla v|^{9/5} dz \\
&\lesssim (R-r)^{-1} R M^{1/2} \cA(R)^{1/10} \cF(R)^{9/10}.
\end{align*}
Hence, we get, by the Young inequality, 
\begin{equation}
\label{E45}
I_2 \lesssim \epsilon \cA(R) + \epsilon^{-1/9} (R/(R-r))^{10/9} M^{5/9} \cF(R).
\end{equation}
\item
We estimate the term $I_3$.
Using the inequality \eqref{E33} in Lemma \ref{T31} in which $v$ is replaced by $\hat{v}$, we obtain
\begin{align*}
I_3
&\lesssim (R-r)^{-1} \left(\fint_{RQ} |p_1 + p_2|^{3/2} dz\right)^{2/3}
\left(\fint_{RQ} |\hat{v}|^3 dz\right)^{1/3} \\
&\lesssim (R-r)^{-1} \fint_{RQ} |\hat{v}|^3 dz.
\end{align*}
Hence, the estimate for $I_3$ is exactly the same as the estimate for $I_2$.
\item
Combining the estimates \eqref{E44} and \eqref{E45} for $I_1$, $I_2$, and $I_3$, we obtain that 
\begin{align*}
\cA(r) + \cE(r)
&\lesssim  \epsilon \cA(R) + \vep^{-1/9}  (R/(R-r))^{20/9} \cF(R) \\
&\quad + \epsilon^{-1/9} (R/(R-r))^{10/9} M^{5/9} \cF(R) + \cP(R).
\end{align*}
By Lemma \ref{T41} we conclude that for all $0<r<R \le 2$
\begin{align*}
\cA(r) + \cE(r)
&\lesssim (R/(R-r))^{20/9} \cF(R) \\
&\quad + (R/(R-r))^{10/9} M^{5/9} \cF(R) + \cP(R).
\end{align*}
Since $M \ge 1$, we get the result \eqref{E42} by choosing $r=R/2$.
\item
We improve integrability of $\nabla v$ using the method in Kinnunen--Lewis \cite{KL}. 
The Calderon--Zygmund stopping time argument plays an important role in proving the reverse H\"older inequality.
We refer the reader Stein's book \cite{S}.
We take the cutoff function $\psi\equiv 1$ in $7/4B$ in the decomposition of pressure $p=p_1 + p_2+p_3$ of \eqref{E24}.
If we set
\[h(t)=\int_{\{|\nabla v|^{9/5}>t\}\cap Q_2} \dist(z,\partial 3/2Q)^{9/2}|\nabla v|^{9/5}dz\]
for the modified distance function
\[\dist(z,\partial 3/2Q) = \inf_{w\in\partial 3/2Q} \set{|z-w|,1/4},\]
with \eqref{E42} and $q=10/9$ in Proposition 5.1 of
Giaquinta--Modica\cite{GM}, we obtain
\[\norm{\nabla v}_{L^{2+\delta}(Q)}
\leq \overline{C} \norm{\nabla v}_{L^{2}(3/2Q)} +
\overline{C}\left(\int_{3/2Q} |\nabla p_3\cdot v|^{(2+\delta)/2}
dz\right)^{1/(2+\delta)}.\] Since $p_3$ is harmonic in $3/2B$, the mean value property gives
\[\sup_{x\in 3/2B} |\nabla p_3| \lesssim \fint_{2B} |p_3| dx.\]
By Jensen's inequality and Young's inequality we obtain that
\begin{align*}
\int_{3/2Q} |\nabla p_3\cdot v|^{(2+\delta)/2} dz
&\le \int_{-9/4}^{9/4} \sup_{x\in 3/2B} |\nabla p_3|^{(2+\delta)/2}
\int_{3/2B} |v|^{(2+\delta)/2} dx dt \\
&\lesssim \int_{-4}^4 \left(\fint_{2B} |p_3| dx\right)^{(2+\delta)/2} \left(\fint_{2B} |v|^{(2+\delta)/2} dx\right) dt \\
&\lesssim \int_{-4}^4 \left(\fint_{2B} |p_3|^{3/2} dx\right)^{(2+\delta)/3} \left(\fint_{2B} |v|^{3(2+\delta)/2(1-\delta)} dx\right)^{(1-\delta)/3} dt \\
&\lesssim \norm{p}_{L^{3/2}(2Q)}^{1+\delta/2}
\norm{v}_{L^{3(2+\delta)/2(1-\delta)}(2Q)}^{1+\delta/2}.
\end{align*}

Recall that $p\in L^{3/2}$ and $v\in L^{10/3}$. If
$\delta<\frac{2}{29}$, then we have
\[\frac{3(2+\delta)}{2(1-\delta)} < \frac{10}{3}\]
and the righthand side is bounded. 
\end{enumerate}

This completes the proof of Theorem \ref{T11}.

\begin{rem}
Seregin \cite{SER} obtained a different version of reverse H\"older inequalities under the assumption $v = \divg b$ and $b \in L^{\infty}(0,T; \BMO)$. 
\end{rem}

\section{Proof of Theorem \ref{T12}}
\label{S5}

We begin by recalling a very well-known lemma about the Hausdorff measure of an upper density of a locally integrable function.

\begin{lem}
\label{T51}
Let $f \in L_{\loc}^1(\R^d)$ and $0 < \alpha < d$.
Denote 
\[E_\alpha(x,r) = r^{-\alpha} \int_{B(x,r)} |f| dy,\]
Then 
\[\cH^\alpha\set{\limsup_{r\to0} E_\alpha(x,r) > 0}=0.\]
\end{lem}

\begin{proof}
For reader's convenience, we give a sketch of the proof.
Fix a compact set $K$ in an open unit cube $Q$ and $n \in \N$.
Set
\[F_n = K \cap \set{\limsup_{r\to0} E_\alpha(x,r) > 1/n}.\]
Fix $\delta<\dist(K,Q^c)$.
For each $x \in F_n$ there exists $r<\delta/5$ such that 
\[E_\alpha(x,r) > 1/(2n).\]
By Vitali's covering lemma there exists countable disjoint balls $B(x_j,r_j)$ such that $F_n \subset \bigcup B(x_j,5r_j)$.  
Since
\[\sum r_j^\alpha \lesssim \int_{\bigcup B(x_j,r_j)} |u| dy\]
and 
\[\sum r_j^d \lesssim \delta^{d-\alpha} \sum r_j^\alpha \lesssim \delta^{d-\alpha} \int_Q |u| dy,\]
we have $\cH^\alpha(F_n) = 0$.
Since $K$ and $n$ are arbitrary, we get the result.
\end{proof}

We need one more elementary lemma.

\begin{lem}
\label{T52}
The condition 
\begin{equation}
\label{E51}
N=\norm{v}_{L^\infty(0,T; \wL)}<\infty,
\end{equation}
implies that for almost all $0 \le t \le T$ and for all $0 < q < 3$, $x \in \R^3$, and $0<R<\infty$
\begin{equation}
\label{E52}
R^{q-3} \int_{B(x,R)} |v(y,t)|^q dy \lesssim \left(\frac{q}{3-q}\right)^{q/3} N^q.
\end{equation}
\end{lem}

\begin{proof}
We denote the level set of $v$ at the height $h \in \R_+$ by 
\[E_t(h) = \set{y\in\Omega:h<|v(y,t)|}.\]
From the condition \eqref{E51} we have for almost all $0 \le t \le T$ and all $h \in \R_+$ 
\[h^3 \m(E_t(h)) \le N^3.\]
Since we have for all $H \in \R_+$
\begin{align*}
\int_{B(x,R)} |v(y,t)|^q dy 
&= q \intzi h^{q-1} \m [B(x,R) \cap E_t(h)] dh \\
&\lesssim q \int_0^H h^{q-1} R^3 dh + q \int_H^\infty h^{q-4} N^3 dh \\
&\lesssim R^3 H^q + \frac{q}{3-q} N^3 H^{q-3},
\end{align*}
we obtain the estimate \eqref{E52} by taking 
\[H = \left(\frac{q}{3-q}\right)^{1/3} NR^{-1}.\]
\end{proof}

Now, we are ready to prove Theorem \ref{T12}, which is, in fact, a direct consequence of Lemma \ref{T51} and Lemma \ref{T52}.
From the estimate \eqref{E52} with $q=2$, we can apply Theorem \ref{T11} so that  
\begin{equation}
\label{E53}
\nabla v \in L_{\loc}^{2+\delta}(\Omega_T)
\end{equation}
for some $\delta>0$.
By H\"older's inequality we have 
\[r^{-1} \int_{Q(z,r)} |\nabla v|^2 dz \lesssim \left(r^{-1+2\delta} \int_{Q(z,r)} |\nabla v|^{2+\delta} dz\right)^{2/(2+\delta)}.\]
The Caffarelli--Kohn--Nirenberg regularity theorem \cite{CKN} implies that 
\[\cS \subset \set{z : \limsup_{r\to0} r^{-1+2\delta} \int_{Q(z,r)} |\nabla v|^{2+\delta} dz > 0}.\]
Therefore, $\cH^{1-2\delta}(\cS)=0$ by Lemma \ref{T51}.
This completes the proof of Theorem \ref{T12}.

\section{Proof of Theorem \ref{T13}}
\label{S6}

We divide the proof several steps.

\begin{enumerate}[\bf{Step} 1)]
\item
We first claim that for any $\gamma \in (0,\delta/10)$ and $z=(x,t) \in \Omega_T$
\begin{equation}
\label{E61}
v \in L^{4+2\delta-\gamma}(Q(z,R)).
\end{equation}
More precisely, we shall show that 
\begin{equation}
\label{E62}
\int_{Q(z,R)} |v|^{4+2\delta-\gamma} dz \lesssim_{\delta,\gamma} R^{\gamma} N^{2+\delta-\gamma} \int_{Q(z,R)} |\nabla v|^{2+\delta} dz + R^{1-2\delta+\gamma} N^{4+2\delta-\gamma}
\end{equation}
where the implied constant can be found explicitly.
In order to see this, we use H\"older's inequality to write 
\begin{align*}
&\int_{B(x,R)} |v|^{4+2\delta-\gamma} dy 
=\int_{B(x,R)} |v|^{2+\delta-\gamma} |v|^{2+\delta} dy \\
&\le \left(\int_{B(x,R)} |v|^{3-3\gamma/(2+\delta)} dy\right)^{(2+\delta)/3}
\left(\int_{B(x,R)} |v|^{3(2+\delta)/(1-\delta)} dy\right)^{(1-\delta)/3}.
\end{align*}
Using \eqref{E52} with $q=3-3\gamma/(2+\delta)$, we estimate the first integral on the right as 
\[\left(\int_{B(x,R)} |v|^{3-3\gamma/(2+\delta)} dy\right)^{(2+\delta)/3} \lesssim \left(\frac{2+\delta}{\gamma}+\frac{\gamma}{2+\delta}-2\right)^{(2+\delta)/3} R^{\gamma} N^{2+\delta-\gamma}.\]
Thus, we have 
\[\int_{B(x,R)} |v|^{4+2\delta-\gamma} dy \lesssim R^{\gamma} N^{2+\delta-\gamma} \left(\int_{B(x,R)} |v|^{3(2+\delta)/(1-\delta)} dy\right)^{(1-\delta)/3}.\]
Since $3(2+\delta)/(1-\delta)$ is the Sobolev exponent of $2+\delta$, we apply the Sobolev--Poincar\'e inequality to get 
\begin{align*}
&\left(\int_{B(x,R)} |v|^{3(2+\delta)/(1-\delta)} dy\right)^{(1-\delta)/3} \\
&\lesssim \left(\int_{B(x,R)} |v-(v)_R|^{3(2+\delta)/(1-\delta)} dy\right)^{(1-\delta)/3} + R^{1-\delta} |(v)_R|^{2+\delta} \\
&\lesssim \int_{B(x,R)} |\nabla v|^{2+\delta} dy + R^{1-\delta} |(v)_R|^{2+\delta}
\end{align*}
where $(v)_R = \fint_{B(x,R)} v(y,t) dy$.
Finally, using the Jensen inequality and \eqref{E52} with $q=2+\delta$, we obtain 
\[R^{1-\delta} |(v)_R|^{2+\delta} \lesssim \left(\frac{2+\delta}{1-\delta}\right)^{(2+\delta)/3} N^{2+\delta} R^{-1-2\delta}.\]
Integrating in time yields the estimate \eqref{E62}.
\item
We fix $\psi$ satisfying $\psi(y) = 1$ for $y \in B(x,R/2)$ and $\psi(y) = 0$ for $y \notin B(x,R)$.
We now use the decomposition of a localized pressure \eqref{E31} and \eqref{E32} with this $\psi$, that is,
\[p\psi = p_1 + p_2 + p_3\]
where 
\begin{align*}
p_1(x,t)
&=|\overline{v}|^2\psi(x,t)+ \frac{3}{4\pi} \int \partial_{y_i} \partial_{y_j} \left(\frac{1}{|x-y|}\right)(\overline{v}_i \overline{v}_j\psi)(y,t) dy, \\
p_2(x,t) &= \frac{3}{2\pi} \int \frac{x_i-y_i}{|x-y|^3}
(\overline{v}_i\overline{v}_j\partial_j\psi)(y,t) dy
+ \frac{3}{4\pi} \int \frac{1}{|x-y|} (\overline{v}_i\overline{v}_j\partial_i\partial_j\psi)(y,t) dy, \\
p_3(x,t) &= \frac{3}{4\pi} \int \frac{1}{|x-y|} (p\Delta \psi)(y,t) dy 
+ \frac{3}{2\pi} \int \frac{x_i-y_i}{|x-y|^3} (p\partial_i\psi)(y,t) dy.
\end{align*}
We notice that $p_1$ and $p_2$ involve $v$ only.
Since we have $v \in L^{4+2\delta-\gamma}(Q(z,R))$, \eqref{E61} in the previous step, we see that
\begin{equation}
\label{E63}
p_1 + p_2 \in L^{2+\delta-\gamma/2}(Q(z,R/5))
\end{equation}
by a direct consequence of $L^q$-continuity of singular integral operators and potential estimates. 
On the contrary, the representation of $p_3$ strongly depends on the outward data of $p$, so we don't expect that $p_3$ gain such a higher integrability in time.
But, since $p_3$ is harmonic in $B(x,R/4)$, we have 
\begin{equation}
\label{E64}
p_3 \in L^{5/3}(t-R^2/25,t+R^2/25 : L^\infty(B(x,R/5))).
\end{equation}
\item
In the previous steps we showed that the weak solution is locally higher integrable.
Considering scaled functional with higher exponent, we can get a better bound for the size of the singular set.
Indeed, it is a consequence of scaling structure of weak solutions.
We have for all $0 < r < R/5$, by H\"older's inequality, 
\begin{align*}
r^{-2} \int_{Q(z,r)} |v|^3 dz 
&\lesssim \left(r^{-1+2\delta-\gamma} \int_{Q(z,r)} |v|^{4+2\delta-\gamma} dz\right)^{3/(4+2\delta-\gamma)}, \\
r^{-2} \int_{Q(z,r)} |p_1+p_2|^{3/2} dz 
&\lesssim \left(r^{-1+2\delta-\gamma} \int_{Q(z,r)} |p_1+p_2|^{2+\delta-\gamma/2} dz\right)^{3/(4+2\delta-\gamma)}, 
\end{align*}
and, by Jensen's inequality,  
\begin{align*}
r^3 \fint_{Q(z,r)} |p_3|^{3/2} dz 
&\lesssim r^3 \fint_{t-r^2}^{t+r^2} \left(\fint_{B(x,r)} |p_3|^{15/(7-6\delta+3\gamma)} dx\right)^{(7-6\delta+3\gamma)/10} dt \\
&\lesssim r^3 \left(\fint_{t-r^2}^{t+r^2} \left(\fint_{B(x,r)} |p_3|^{15/(7-6\delta+3\gamma)} dx\right)^{(7-6\delta+3\gamma)/9} dt\right)^{9/10} \\
&= \left(r^{-1+2\delta-\gamma} \int_{t-r^2}^{t+r^2} \left(\int_{B(x,r)} |p_3|^{15/(7-6\delta+3\gamma)} dx\right)^{(7-6\delta+3\gamma)/9} dt\right)^{9/10}.
\end{align*}
Combining those estimates, we obtain 
\begin{align*}
&r^{-2} \int_{Q(z,r)} |v|^3 +|p|^{3/2} dz \\
&\lesssim \left(r^{-1+2\delta-\gamma} \int_{Q(z,r)} |v|^{4+2\delta-\gamma} + |p_1+p_2|^{2+\delta-\gamma/2} dz\right)^{3/(4+2\delta-\gamma)} \\
&\quad + \left(r^{-1+2\delta-\gamma} \int_{t-r^2}^{t+r^2} \left(\int_{B(x,r)} |p_3|^{15/(7-6\delta+3\gamma)} dx\right)^{(7-6\delta+3\gamma)/9} dt\right)^{9/10}.
\end{align*}
There exists a positive number $\vep$ such that if 
\[\left(\int_{Q(z,r)} |v|^{4+2\delta-\gamma} + |p|^{2+\delta-\gamma/2} dz + \int_{t-r^2}^{t+r^2} \left(\int_{B(x,r)} |p_3|^{15/(7-6\delta+3\gamma)} dx\right)^{(7-6\delta+3\gamma)/9} dt\right) < \vep r^{1-2\delta+\gamma}\]
for some $r < R/5$, then $z$ is a regular point by the well-known $L^3$-regularity criterion.
\item
In the previous setp, we obtain that for each $z=(x,t) \in \cS \cap Q(z_0,R/10)$ we should have for all $0 < r < R/50$ 
\begin{equation}
\label{E65}
\begin{split}
\vep r^{1-2\delta+\gamma} 
&\le \int_{Q(z,r)} |v|^{4+2\delta-\gamma} + |p|^{2+\delta-\gamma/2} dz \\
&\quad + \int_{t-r^2}^{t+r^2} \left(\int_{B(x,r)} |p_3|^{15/(7-6\delta+3\gamma)} dx\right)^{(7-6\delta+3\gamma)/9} dt.
\end{split}
\end{equation}
Fix $r < R/50$ and consider the covering $\set{Q(z,r) : z \in \cS \cap Q(z_0,R/10)}$.
Adapting the argument in \cite{MR2864801}, we can choose a finite disjoint sub-family 
\[\set{Q(z_j,r) : j=1,2,\dots,J(r)}\]
such that $\cS \cap Q(z_0,R/10) \subset \bigcup Q(z_j,5r)$ by the Vitali covering lemma.
We need to pay more attention in order to use the disjointness of the subcovering.
Summing the inequality \eqref{E65} at $z_j$ for $j=1,2,\dots,J$ yields
\begin{align*}
J \vep r^{1-2\delta+\gamma}
&\le \sum_{j=1}^{J} \int_{Q(z_j,r)} |v|^{4+2\delta-\gamma} + |p|^{2+\delta-\gamma/2} dz \\
&\quad + \sum_{j=1}^{J} \int_{t_0-R^2/25}^{t_0+R^2/25} \left(\int_{B(x_j,r)} |p_3|^{15/(7-6\delta+3\gamma)} dx\right)^{(7-6\delta+3\gamma)/9} dt \\
&\le \int_{Q(z_0,R/5)} |v|^{4+2\delta-\gamma} + |p|^{2+\delta-\gamma/2} dz \\
&\quad + \int_{t_0-R^2/25}^{t_0+R^2/25} \sum_{j=1}^{J} \left(\int_{B(x,r)} |p_3|^{15/(7-6\delta+3\gamma)} dx\right)^{(7-6\delta+3\gamma)/9} dt 
\end{align*}
For the last sum, we use Young's inequality, that is, for $\theta \in [0,1]$ 
\[\sum_{j=1}^{J} a_j^\theta \le J^{1-\theta} \left(\sum_{j=1}^{J} a_j\right)^\theta.\]
Due to the fact that $p_3$ is harmonic in $B(z_0,R/5)$ and the estimates \eqref{E63} and \eqref{E64}, there are positive numbers $A_1, A_2 \in \R_+$ such that 
\begin{equation}
\label{E66}
\begin{split}
J \vep r^{1-2\delta+\gamma}
&\le \int_{Q(z_0,R/5)} |v|^{4+2\delta-\gamma} + |p|^{2+\delta-\gamma/2} dz \\
&\quad + J^{(2+6\delta-3\gamma)/9} \int_{t_0-R^2/25}^{t_0+R^2/25} \left(\int_{B(x_0,R/5)} |p_3|^{15/(7-6\delta+3\gamma)} dx\right)^{(7-6\delta+3\gamma)/9} dt \\
&=: A_1+A_2 J^{(2+6\delta-3\gamma)/9}.
\end{split}
\end{equation}
\item
The minimum number of parabolic cylinders $Q(z,r)$ required to cover the set $\cS \cap Q(z,R/5)$ is less than or equal to $J$.
We rewrite \eqref{E66} as 
\[J \vep r^{C_1} \le A_1 + A_2 J^{C_2}\]
where $C_1=1-2\delta+\gamma$ and $C_2=(2+6\delta-3\gamma)/9$.
Applying Young's ineaulity to $A_2 J^{C_2}$, we get 
\[J \vep r^{C_1} \le A_1 + (1-C_2) [A_2 (\vep r^{C_1})^{-C_2}]^{1/(1-C_2)} + C_2 J \vep r^{C_1}\]
and so 
\[J(r) \lesssim A_1 \vep^{-1} r^{-C_1} + A_2^{1/(1-C_2)} \vep^{-1/(1-C_2)} r^{-C_1/(1-C_2)}.\]
Therefore, by an elementary calculation we get 
\[\limsup_{r\to0} \frac{\log J(r)}{-\log r} = \frac{C_1}{1-C_2} = \frac{9-18\delta+9\gamma}{7-6\delta+3\gamma}.\]
Since $\delta$ and $\gamma$ can be arbitrarily close to $\delta_0$ and $0$ respectively, we conclude the assertion \eqref{E15}.
\end{enumerate}
This completes the proof of Theorem \ref{T13}.

%

\end{document}